\documentclass[12pt]{amsart}

\usepackage{amsfonts,amsmath,amssymb,amsthm}
\usepackage{graphicx}
\usepackage[latin1]{inputenc}
\usepackage{ifpdf}

\ifpdf
\fi

\newcommand{\Cc}{\mathbb{C}}
\newcommand{\C}{\mathbb{C}}

\newcommand{\defi}[1]{\emph{#1}}

\renewcommand{\epsilon}{\varepsilon}

\renewcommand {\leq}{\leqslant}
\renewcommand {\geq}{\geqslant}

\newcommand{\grad}{\mathop{\mathrm{grad}}\nolimits}
\renewcommand{\Re}{\mathop{\mathrm{Re}}\nolimits}

{\theoremstyle{plain}
\newtheorem{theorem}{Theorem}    % theorem with number
\newtheorem{lemma}[theorem]{Lemma}        % lemma with number
\newtheorem{proposition}[theorem]{Proposition}        % lemma with number
\newtheorem{corollary}[theorem]{Corollary}        % lemma with number
{\theoremstyle{definition}
\newtheorem{definition}[theorem]{Definition}
\newtheorem*{remark*}{Remark}  % theorem without number
\newtheorem{example}{Example}
}

\begin{document}
%%%%%%%%%%%%%%%%%%%%%%%%%%%%%%%%%%%%%%%%%%%%%%%%%%%%%%%%%%%%%%%%%
\title{Milnor fibrations of meromorphic functions}

\author{Arnaud Bodin, Anne Pichon, Jos\'e Seade}



\date{\today}

\begin{abstract}
In analogy with the holomorphic case, we compare the topology of
 Milnor fibrations associated   to a meromorphic germ
$f/g$ : the local Milnor fibrations given on Milnor tubes over
punctured discs around the critical values of $f/g$, and the
Milnor fibration on a sphere.
\end{abstract}

\maketitle

%%%%%%%%%%%%%%%%%%%%%%%%%%%%%%%%%%%%%%%%%%%%%%%%%%%%%%%%%%%%%%%%%
%%%%%%%%%%%%%%%%%%%%%%%%%%%%%%%%%%%%%%%%%%%%%%%%%%%%%%%%%%%%%%%%%
\section{Introduction}

The classical fibration theorem of Milnor in \cite {Mi}
  says that every holomorphic map (germ) $f:(\Cc^n,0)
\to (\Cc,0)$  with a critical point at $0 \in \Cc^n$   has  two
naturally associated fibre bundles, and both of these are
equivalent. The first  is:
\begin{equation}\label{Mil-fib-hol}
\phi = \frac{f}{\vert f \vert} : \mathbb S_\epsilon \setminus K
\longrightarrow \mathbb S^1 \end{equation} where $\mathbb
S_\epsilon$ is a sufficiently small sphere around $0 \in \Cc^n$
and $K = f^{-1}(0) \cap \mathbb S_\epsilon$ is the link of $f$ at
$0$. The second fibration is:
\begin{equation}\label{Mil-Le-fib-hol}
 {f}  : \mathbb B_\epsilon \cap f^{-1}(\partial \mathbb D_\delta)
\longrightarrow \partial \mathbb D_\delta \cong \mathbb S^1
\end{equation}
where $\mathbb B_\epsilon$ is the closed ball in
$\Cc^n$ with boundary $\mathbb S_\epsilon$ and $\mathbb D_\delta$
is a disc around $0 \in \Cc$ which is sufficiently small with
respect to $\epsilon$.

The set $N(\epsilon,\delta) =  \mathbb B_\epsilon \cap
f^{-1}(\partial \mathbb D_\delta)$ is usually called a local {\it
Milnor tube} for $f$ at $0$, and it is diffeomorphic to $\mathbb
S_\epsilon$ minus an open regular neighbourhood $T$ of $K$. (Thus,
to get the equivalence of the two fibrations  one  has to ``extend"
the latter fibration to $T \setminus K$.) In fact, in order to
have the second fibration one needs to know that every map-germ
$f$ as above has the so-called ``Thom property", which was not
known when Milnor wrote his book. What he proves is that the
fibers in (\ref{Mil-fib-hol}) are diffeomorphic to the
intersection $f^{-1}(t) \cap \mathbb B_\epsilon$ for $t$ close
enough to $0$. The statement that
 (\ref{Mil-Le-fib-hol}) is a fibre bundle was proved later in \cite{Le} by L\^e D\~ung
 Tr\'ang in the more general setting of holomorphic maps defined
 on arbitrary complex analytic spaces, and we call it the {\it
 Milnor-L\^e} fibration of $f$. Once we know that
 (\ref{Mil-Le-fib-hol}) is a fibre bundle,   the arguments of
 \cite[Chapter 5]{Mi}   show this is equivalent to the
{\it Milnor} fibration (\ref{Mil-fib-hol}).

 The literature about these fibrations is vast, and so are their
 generalizations to various   settings, including real
 analytic map-germs and meromorphic maps, and that is the starting
 point of this article.

\bigskip

Let  $U$ be an open neighbourhood of $0$ in $\Cc^n$ and let $f,g :  U \longrightarrow \Cc$
be two holomorphic functions without common factors such that $f(0) = g(0)=0$.

Let us consider the meromorphic function $F = f/g : U \rightarrow
{\mathbb C P}^1 $ defined by $(f/g) (x) =  [f(x)/ g(x)]$. Notice
that $f/g$ is not defined on the whole $U$ ; its indetermination
locus is
$$I= \big\{ z \in U \mid f(x)=0 \text{ and } g(x)=0 \big\} \,.$$
In particular,  the fibers of $F=f/g$ do not contain any point of $I$ : for each $c \in \C$,
the fiber $F^{-1}(c)$  is the set
$$F^{-1}(c)= \big\{x \in U  \mid f(x)-cg(x) = 0\big\} \setminus I \,.$$

In a series of articles, S. M. Gusein-Zade, I.  Luengo and A.
Melle-Hern\'andez, and later D. Siersma and M. Tib\v ar, studied
local Milnor fibrations of the type (\ref{Mil-Le-fib-hol})
associated to every critical value of the meromorphic map $F =
f/g$. See for instance \cite{GLM1, GLM2}, or Tibar's book \cite
{Tib} and the references in it. Of course the ``Milnor tubes"
$\mathbb B_\epsilon \cap F^{-1}(\partial \mathbb D_\delta)$
in this case are not actual tubes in general, since they may
contain $0 \in U$ in their closure. These are in fact ``pinched
tubes".

It is thus natural to ask whether one has for meromorphic
map-germs fibrations of Milnor type (\ref{Mil-fib-hol}), and if
so, how these are related to those of the Milnor-L\^e type
(\ref{Mil-Le-fib-hol}) studied (for instance) in \cite {GLM1,
GLM2, Tib}. The first of these questions was addressed in
\cite{P2, BP, PS} from two different viewpoints, while the answer to
the second question is the bulk of this article.

In fact, it is proved in \cite{BP} that if the meromorphic germ
$F=f/g$ is semitame  (see the definition in Section
\ref{Section:semitame}), then
\begin{equation}
\label{FibrationBP}
\frac{F}{|F|} = \frac{f/g}{\vert f/g \vert} :
\mathbb S_\epsilon \setminus (L_f \cup L_g) \longrightarrow \mathbb S^1
\end{equation}
is a fiber bundle, where $L_{f} = \{ f=0\} \cap \mathbb
S_{\epsilon}$ and $L_{g} = \{ g=0\} \cap \mathbb S_{\epsilon}$ are
the oriented links of $f$ and $g$.
Notice that away from the link   $L_f \cup L_g$ one has an
equality of maps:
$$ \frac{f/g}{\vert f/g \vert} \, = \, \frac{f \bar g}{\vert f \bar g
\vert}\,,$$ where $\bar g$ denotes complex conjugation. It is
proved in \cite{PS} that if the real analytic map $f \bar g$ has
an isolated critical value at $0 \in \Cc$ and satisfies the Thom
property, then  the Milnor-L\^e fibration of $f \bar g$,
\begin{equation}
\label{FibrationPS}
N(\epsilon,\delta) := [\mathbb B_\epsilon \cap (f \bar
g)^{-1}(\partial \mathbb D_\delta)] \buildrel{f \bar g} \over
{\longrightarrow}
\partial \mathbb D_\delta \cong \mathbb S^1 \,,
\end{equation}
is equivalent to the Milnor fibration (\ref{FibrationBP}) of $f/g$ when this map is
semitame. That is, the fibration (\ref{FibrationPS}) on the
Milnor tube $N(\epsilon,\delta)$ of $f \bar g$ is equivalent to
the Milnor fibration (\ref{FibrationBP}) of the meromorphic germ
$f/g$.

In this article we complete the picture by comparing  the local
fibrations of Milnor-L\^e type of a meromorphic germ $f/g$ studied
by Gusein-Zade {\it et al}, with the Milnor fibration
(\ref{FibrationBP}). We prove that if the germ $f/g$ is
semitame and (i)-tame (see Sections \ref{Section:semitame} and \ref{sec:itame}), then the global
Milnor fibration (\ref{FibrationBP}) for $f/g$ is obtained from
the local Milnor fibrations of $f$ at $0$ and $\infty$ by a gluing
process that is, fiberwise, reminiscent of the classical connected
sum of manifolds (see Theorem \ref{Result}, and its corollaries,
in Section \ref{Section:results}).

\bigskip

{\small
\emph{Acknowledgements :} The research for this article
was partially supported by the CIRM at Luminy, France, through a
``Groupe de Travail"; there was also partial financial support from
the Institute de Math\'ematiques de Luminy and from CONACYT and
PAPIIT-UNAM, Mexico, and the authors are grateful to all these
institutions for their support.
}

%%%%%%%%%%%%%%%%%%%%%%%%%%%%%%%%%%%%%%%%%%%%%%%%%%%%%%%%%%%
%%%%%%%%%%%%%%%%%%%%%%%%%%%%%%%%%%%%%%%%%%%%%%%%%%%%%%%%%%%
\section {Semitameness and the global Milnor fibration of $F$}
\label{Section:semitame}

Adapting  Milnor's definition \cite{Mi}, we  define the gradient
of $F=f/g$ at a point $x \in U \setminus I$  by :
$$\grad (f/g) = \left(\overline{\frac{\partial (f/g)}{\partial x_1}}, \ldots,
  \overline{\frac{\partial (f/g)}{\partial x_n}}\right).$$
The following definitions were introduced in \cite{BP} following
ideas of \cite{NZ}. We consider the set

 $$M(F) = \big\lbrace x \in U \setminus I \mid \exists
  \lambda \in \Cc, \grad (f/g)(x)=\lambda x
\big\rbrace$$
 consisting of  the points of non-transversality between
the fibres  of $f/g$ and the spheres ${\mathbb S}_{r}$
centered  at the origin of $\Cc^n$.

\begin{definition} A value $c \in {\mathbb C P}^1$ is called
\defi{atypical} if there  exists a sequence
 of points $(x_k)_{k\in {\mathbb N}}$ in  $M(F)$ such that
$$  \lim_{k \rightarrow  \infty }x_k = 0 \quad    \hbox{ and  } \quad
\lim_{k \rightarrow \infty} F(x_k) = c\,.$$
Otherwise it is called typical. The  finite set $B$ of the atypical values is
called the \defi{bifurcation set} of the meromorphic function
$f/g$.
\end{definition}

Let $L_{f} = \{ f=0\} \cap \mathbb
S_{\epsilon}$ and $L_{g} = \{ g=0\} \cap \mathbb S_{\epsilon}$ be
the oriented links of $f$ and $g$.

\begin{proposition} \label{fib} Let $W$ be an open set in $\mathbb C P^1$ such that
$W \cap B = \emptyset$.
There exists $\epsilon_0>0$ such that for each $\epsilon \leq
\epsilon_0$, the map
 $$\Phi_W = \frac{f/g}{|f/g|} : \big( \mathbb S_\epsilon \setminus (L_f \cup L_g) \big) \cap F^{-1}(W)  \longrightarrow \mathbb S^1$$
is a $C^{\infty}$ locally trivial fibration.
\end{proposition}

The proof is that of \cite[Theorem 2.6]{BP}; it follows Milnor's
proof \cite[Chapter 4]{Mi} with minor modifications.  See also
\cite{NZ}. The main modification of Milnor's proof concerns Lemma
4.4 of \cite{Mi}, for which an adapted formulation and a detailed
proof is given in \cite[Lemma 2.7]{BP}.

\begin{definition} The meromorphic function $f/g$ is \defi{semitame} at $0$ if
$B \subset \{0, \infty \}$.
\end{definition}

Proposition \ref{fib} is a more general statement than
\cite[Theorem 2.6]{BP}.  When $F$ is semitame, the
following is obtained by applying Proposition \ref{fib} to $W=
\mathbb C P^1 \setminus \{0,\infty\}$ :

\begin{corollary} (\cite[Theorem 2.6]{BP})
If $F$ is semitame, then there exists $\epsilon_0>0$ such that for
each $\epsilon \leq \epsilon_0$,
 the map
 $$\Phi_F =  \frac{f/g}{|f/g|} : \mathbb S^{2n-1}_\epsilon \setminus (L_f \cup L_g)
 \longrightarrow \mathbb S^1$$
is a $C^{\infty}$ locally trivial fibration.
\end{corollary}

\begin{definition}
When $F$ is semitame, we call $\Phi_F$ \defi{the global
Milnor fibration} of the meromorphic germ $F$.
\end{definition}

It is shown in \cite{PS} that $\Phi$ is a fibration of the multilink $L_f \cup -L_g$, where
$ -L_g$ means  $L_g$ with the opposite orientation.

\bigskip

 For our purpose, it  will be necessary to  consider
the restriction $\check{\Phi}_F$ of $\Phi_F$ to $\big( \mathbb S^{2n-1}_\epsilon \setminus (L_f \cup L_g) \big)\setminus F^{-1}(\mathbb D_{\delta}(0) \cup \mathbb D_{R}(0) )$ where $\delta \ll 1$ and  $1 \ll R$.

\begin{definition}
We denote by $\check{\mathcal{M}}_F$ the fibre of $\check{\Phi}_F$ and we call
it the \defi{truncated global Milnor fibre} of $F$.
\end{definition}

%%%%%%%%%%%%%%%%%%%%%%%%%%%%%%%%%%%%%%%%%%%%%%%%%%%%%%%%%%%
%%%%%%%%%%%%%%%%%%%%%%%%%%%%%%%%%%%%%%%%%%%%%%%%%%%%%%%%%%%
\section{Tameness near the indetermination points}
\label{sec:itame}

In this section we introduce a technical condition on $f/g$: the
(i)-tameness ((i) for ``indetermination")  which enables us to
control the behaviour of $f/g$ in a neighbourhood of its
indermination points when $n \geq 3$. This condition will appear
as an essential hypothesis for our main Theorem \ref{th:fib0}. Note
that this section only concerns the case $n \geq 3$.

Let us fix $r >0$ and let us consider some sufficiently small constants
$0<  \epsilon' \ll \delta \ll  \epsilon \ll 1$. These constant will be defined more
precisely in the proof of Theorem \ref{th:fib0}.

 Let $X = F^{-1} \left( {\mathbb D}_r(0)\setminus \mathring {\mathbb D}_\delta(0)  \right)
 \cap \left( {\mathbb B}_\epsilon \setminus  \mathring {\mathbb B}_{\epsilon'} \right)
 $.

 For $\eta >0$,  we consider  the neighbourhood of $I$ defined by:
   $$N_{\eta} =\big\{ z \in \mathbb B_{\epsilon} \ | \ |f(z)|^2 + |g(z)|^2 \leq \eta^2\big\},$$
and its boundary,
$$\partial N_{\eta}     =
   \big\{ z \in \mathbb B_{\epsilon} \ | \ |f(z)|^2 + |g(z)|^2 = \eta^2\big\}.$$

The proof of Theorem \ref{th:fib0} is based on the existence of a
vector field $v$ on X which satisfies for all sufficiently
small $\eta,\, 0 < \eta \ll \epsilon'$ the following properties:

\begin{description}
\item[(1)]  The argument of $f$ is constant along the integral
curves of $v$.
 \item[(2)] The norm of $z$ is strictly increasing
along the integral curves of $v$.
 \item[(3)] For all  $z \in
N_{\eta}$, the integral curve passing through $z$ is contained in
the tube $\partial N_{\eta'}$ where ${\eta'} ^2= |f(z)|^2 +
|g(z)|^2 $.
\end{description}
In this paper, we use two different inner products on $\mathbb C^n$ :
\begin{enumerate}
\item The usual hermitian form $\langle  \ , \
\rangle  : \mathbb C^n \times \mathbb C^n \rightarrow \mathbb C$ defined for $z=(z_1,\ldots,z_n), z'=(z'_1,\ldots,z'_n) \in \mathbb C^n$ by :
$$ \langle  z , z' \rangle = \sum_{k=1}^n z_k\bar{z'_k}$$
\item The usual inner product $\langle  \ , \
\rangle_{\mathbb R}  : \mathbb R^{2n} \times \mathbb R^{2n} \rightarrow \mathbb R$ on $\mathbb R^{2n}$ :
$$ \langle  z , z' \rangle_{\mathbb R} = \sum_{k=1}^n (x_k x'_k + y_k y'_k),$$
where $\forall k, z_k=x_k+iy_k$ and $ z'_k=x'_k+iy'_k$.
\end{enumerate}
Notice that for $z,z' \in \mathbb C^n$,

$$ \langle  z , z' \rangle =  \langle  z , z' \rangle_{\mathbb R} + i  \langle  z , i z' \rangle_{\mathbb R}$$

\medskip

As we will show in the proof of Theorem \ref{th:fib0}, the semitameness of $f/g$ guarantees the existence of a vector
field $v$ on $X$ such that:

\begin{enumerate}
\item[(i)] For all $z\in X, \langle v(z) , \grad \log F(z)
\rangle = +1 $.
 \item[(ii)] For all $z\in X \setminus M(F)$ ,
$\langle v(z) , z \rangle >0$.
 \item[(iii)]
  For all $z \in U$, $Re \langle  v(z),z \rangle >0$".
\end{enumerate}

So that conditions (1) and (2) are satisfied. We now introduce an
additional hypothesis which will ensure that (3) is also
satisfied, {\it i.e.} that $v$ is such that :
\begin{enumerate}
\item[(iv)] For all $z \in X \cap  N_{\eta} \setminus I$ one has
$v(z) \in T_z \partial N_{\eta'}$, where ${\eta'} ^2= |f(z)|^2 +
|g(z)|^2 $.
\end{enumerate}

As shown in the proof of the Theorem \ref{th:fib0},   semitameness is sufficient
to define such a $v$ in a neighbourhood of $M(F) \cap N_{\eta} $.
Now, let $z \in N_{\eta} \setminus M(F)$.

We set $\gamma(z)= |f(z)|^2 + |g(z)|^2$  so that
$$ T_z \partial N_{\eta'} = \{ v \in {\mathbb R}^{2n} \ | \ \langle v , \grad_{\mathbb R}
\gamma (z) \rangle_{\mathbb R} = 0\}\,.$$

Then a vector $v \in {\mathbb R}^{2n}$ satisfies (i), (ii) and (iv) if and only if

$$   \langle v ,  \grad
\log F(z) \rangle = +1 \;,  \langle v , z \rangle >0 \;  \hbox{ and } \;  \langle v , \grad_{\mathbb R}
\gamma (z) \rangle_{\mathbb R} = 0\,.$$

Such a $v$ exists if and only if $\grad_{\mathbb R} \gamma (z)  $
does not belong to the $\mathbb C$-vector space generated by $z$ and
$\grad \log F(z)$, or equivalently by $z$ and $\grad F(z)$. This makes natural the following
definition. We set:

$$N(F)
= \{ z \in U \setminus I \ \big | \  \exists \, \lambda, \mu \in
\mathbb C, \; \grad_{\mathbb R} \gamma (z) = \lambda z + \mu \grad F(z) \} \,. $$

\begin{definition} Let $n \geq 3$. We say that $f/g : (\mathbb C^n,0) \rightarrow (\mathbb C,0)$
is (i)-tame  if  there exist sufficiently small constants $0< \eta
\ll \epsilon' \ll \delta \ll \epsilon \ll 1$ such that $$ \big(
N(F) \cap N_{\eta} \cap X \big) \subset  \big( M(F) \cap N_{\eta}
\cap X \big)\,.$$
 When $n=2$, we define the (i)-tameness as an
empty condition.
\end{definition}

Notice that (i)-tameness is a generic property in the following sense. Let $f , g: (\mathbb C^n,0)
\rightarrow (\mathbb C,0)$ without common branches. Then the set of indetermination points
$I=\{z \in \mathbb C^n \ | \ f(z)=g(z)=0 \}$ has complex dimension $n-2$. Moreover, $N(F) \cup M(F)$
 is included in the set

 $$P(F) =\{ z \in \mathbb C^n \ | \ \mathrm{rank} \, A(z) <3\}\,,$$
where $A(z)$ is the matrix

$$
 \left(
\begin{array}{ccccc}
 \overline{\frac{\partial f}{\partial z_1}} f +g \overline{  \frac{\partial g}{\partial z_1}} \quad &
 \overline{\frac{\partial f}{\partial z_2}} f +g \overline{  \frac{\partial g}{\partial z_2}} &
 \ldots
  &
 \overline{\frac{\partial f}{\partial z_n}} f +g \overline{  \frac{\partial g}{\partial z_n}}   \\

&&& \\
z_1 &    z_2 & \ldots & z_n \\
&&& \\
 \overline{ \frac{\partial f}{\partial z_1} g  -   \frac{\partial g}{\partial z_1}f } \quad &
  \overline{ \frac{\partial f}{\partial z_2} g -  \frac{\partial g}{\partial z_2}f } &
   \ldots &
    \overline{ \frac{\partial f}{\partial z_n} g -  \frac{\partial g}{\partial z_n}f }\\
\end{array}
\right)
$$

Then $P(F)$ is generically a real analytic submanifold of $\mathbb C^n$ with real codimension $2n-4$.
Then generically, the two germs of analytic submanifolds $(I,0)$ and $(N(F)\cup M(F),0)$
intersect only at $0$. Therefore, when the constants
$0<  \eta  \ll \epsilon' \ll \delta \ll  \epsilon \ll 1$ are sufficiently small, we obtain
$P(F) \cap N_{\eta} \cap X  = \emptyset$, and then, $f/g$ is (i)-tame.

\medskip

%%%%%%%%%%%%%%%%%%%%%%%%%%%%%%%%%%%%%%%%%%%%%

\begin{example}
It may happen that $f/g$ is (i)-tame even if $I$ is contained in $ N(F)\cup
M(F)$. For example, let $f ,g : (\mathbb C^3,0)\rightarrow (\mathbb
C,0)$ be defined by $f(x,y,z)=x^p$ and $g(x,y,z)=y^q$. Then the set of indetermination
points of $f/g$   is the
$z$-axis, and the set $P(f/g)$ has equation $\det A(x,y,z) = 0$, {\it i.e.} :
$$zx^{p-1}y^{q-1}(|x|^{2p}+|y|^{2q})=0\,.$$
Then $N(f/g)$ is included in the plane $\{z=0\}$ and $f/g$ is
(i)-tame, whereas $I \subset P(f/g)$. Hence $f/g$ is also
semitame.
\end{example}

\begin{example}
Let $f=f(x,y )$ and $g=g(x,y)$ be considered as germs from $(\Cc^3,0)$ to $(\Cc,0)$. Then the set of indetermination
points of $f/g$ is again the
$z$-axis, and the set $P(f/g)$ has equation
$$ z \left(   \frac{\partial f}{\partial  y}   \frac{\partial g}{\partial  x}
- \frac{\partial f}{\partial  x}   \frac{\partial g}{\partial  y}
\right) \big(  |f|^2 + |g|^2   \big) =0\,.$$
Therefore $f/g$ is (i)-tame  if and only if the jacobian curve
$\{\frac{\partial f}{\partial  y}   \frac{\partial g}{\partial  x} -
 \frac{\partial f}{\partial  x}   \frac{\partial g}{\partial  y}
=0\}$ of the germ
$(f,g) : (\mathbb C^2,0) \rightarrow (\mathbb C^2,0)$ is included in the curve $\{fg=0\}$.

On the other hand, it is easy to obtain examples with $f/g$
semitame. For instance, with $f,g$ as above, if we regard  $f/g$ as a map-germ at $0 \in \mathbb C^3$, then this    is semitame if
$f/g$ is semitame as a germ from $(\mathbb C^2,0)$ into $(\mathbb
C,0)$, since a sequence of  bad
points $(z_k)$ for $f/g$ would project on the plane $z=0$ to a
sequence of bad points for $f/g:(\mathbb C^2,0) \to (\mathbb C,0)$. Now,
it is easy to check whether $f/g:(\mathbb C^2,0) \to (\mathbb C,0)$ is
semitame by using the characterization of semitameness given in
\cite[Theorem 1]{BP} when $n=2$ :  $f/g$ is semitame if and only if
the multilink $L_f \cup -L_g$  is fibered. This latter condition
is easily checked by computing a resolution graph of the
meromorphic function $f/g$: the multilink $L_f \cup -L_g$  is fibered if and
only if the multiplicities of $f$ and $g$ are different on each
rupture component of the exceptional divisor of $f/g$.

\end{example}

\medskip
\begin{example}
Let $f(x,y ) = x^3 + y^2$ and $g(x,y) =x^2 + y^3$. Then $f/g$ is
semitame, as can be seen on  the resolution graph of $f/g$
represented on Figure 1. The number between parenthesis on each
vertex is the difference $m_f-m_g$ where $m_f$ (respectively
$m_g$) is the multiplicity of $f$ along the corresponding
component of the exceptional divisor.

\begin{figure}
\includegraphics{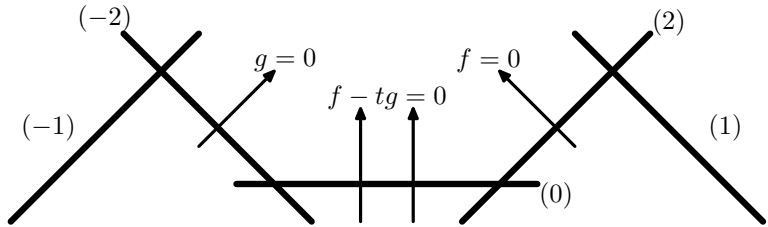}
\caption{\label{fig:shortexample} Resolution graph of $x^3+y^2/x^2+y^3$}
\end{figure}

But $f/g$, seen as a map in variables $(x,y,z)$ is not (i)-tame,
because the germ of Jacobian curve $(J,0)$ of $(f,g)$ has equation $xy=0$ and $N(f/g) = J \setminus I$.
\end{example}

%%%%%%%%%%%%%%%%%%%%%%%%%%%%%%%%%%%%%%%%%%%%%%%%%%%%%%%%%%%
%%%%%%%%%%%%%%%%%%%%%%%%%%%%%%%%%%%%%%%%%%%%%%%%%%%%%%%%%%%
\section{The local Milnor fibrations of $F$}

The local Milnor fibers of a meromorphic function $F$ were
defined in \cite{GLM1} as follows. Let us fix $c \in \mathbb C
P^1$. There exists $\epsilon_0 >0$ such that for any $\epsilon$,
$0< \epsilon \leq \epsilon_0$,  the restriction $F_\mid: \mathbb
B_{\epsilon} \setminus (f=0)\cup(g=0) \rightarrow \mathbb C P^1$
defines a $C^{\infty}$ locally trivial fibration over a punctured
neighbourhood $\Delta_c$ of the point $c$ in $\mathbb C P^1$.

\begin{definition}
The fiber $ {\mathcal{M}}_F^c = F^{-1}(c') \cap \mathbb B_{\epsilon},  c' \in \Delta_c$ of
this fibration is called the \defi{$c$-Milnor fiber} of $F$.

Notice that $ {\mathcal{M}}_F^c $ is a non-compact $n$-dimensional manifold with boundary.

Let $\delta$, $0< \delta \ll \epsilon$, be such that $\mathbb
D_{\delta}(c) \subset \Delta_c \,.$  We call the restriction
$$ \phi_c = F_{\mid } : F^{-1} ({\mathbb S}^1_\delta(c)) \cap B_\epsilon \longrightarrow {\mathbb S}^1_\delta(c)$$
the \defi{$c$-local Milnor fibration} of the meromorphic map $F$.
\end{definition}

According to \cite[Lemma 1]{GLM1}, the diffeomorphism class of the
non-compact $n$-complex manifold $ {\mathcal M}_F^c$ does not depend
on $\epsilon$, and the isomorphism class of the fibration $\phi_c$
does not depend on $\epsilon$ and $\delta$. As shown in
\cite{Tib}, this is in fact an immediate consequence of L\^e's
fibration theorem in \cite{Le} applied to the pencil $\{f - tg = 0
\}$.

\medskip

For our purpose, it  will be necessary to  consider the
restriction of $\phi_c$ to the complement in ${\mathbb B}
_{\epsilon}$ of a small ball ${\mathbb B} _{\epsilon'}$, $0 \ll
\epsilon' \ll \delta$, defined as follows. Since $\Delta_c \cap B
= \emptyset$, there exists
 $\epsilon'$, $0<\epsilon' \ll \delta \ll \epsilon$, such that
 $M(F)\cap F^{-1}({\mathbb S}^1_\delta(c)) \cap
  {\mathbb  B}_{\epsilon'} =\emptyset$. For such an  $\epsilon'$, we  consider the  restriction of the
  $c$-local Milnor fibration

$$ \check{\phi}_c = F_{\mid } : F^{-1} (\mathbb S^1_\delta(c)) \cap (\mathbb  B_\epsilon \setminus
\mathring{\mathbb B}_{\epsilon'} ) \longrightarrow \mathbb
S^1_\delta(c)\,.$$ And we denote by $\check{\mathcal  M}_F^c =  {\mathcal
M}_F^c \setminus  \mathring{\mathbb B}_{\epsilon'}$ the fiber of
$\check{\phi}_c$

Again, the diffeomorphism class of $\check{{\mathcal M}}_F^c$ and the
isomorphism class of $\check{\phi}_c$ do not depend on $\epsilon,
\delta$ and $\epsilon'$.

%%%%%%%%%%%%%%%%%%%%%%%%%%%%%%%%%%%%%%%%%%%%%%%%%%%%%%%%%%%
%%%%%%%%%%%%%%%%%%%%%%%%%%%%%%%%%%%%%%%%%%%%%%%%%%%%%%%%%%%
\section{The results}
\label{Section:results}

\begin{theorem}
\label{Result}
\label{th:fib0}
 Let $f,g : (\mathbb C^n,0) \rightarrow (\mathbb
C,0)$ be two germs of holomorphic functions without common
branches such that $F=f/g$ is  (i)-tame.
 Then for all $r
\in \mathbb C P^1 \setminus \{0,\infty\}$ such that $B \cap {\mathbb D}_r(0)= \{0\}$, there exist  $\epsilon,
\epsilon'$ and $\delta$, $0<  \epsilon' \ll \delta \ll  \epsilon \ll
1$, such that the restricted $0$-local Milnor fibration
\begin{equation}
\label{eq:f0} \check{\phi_0} : F^{-1} (\mathbb S^1_\delta(0)) \cap
(\mathbb B_\epsilon \setminus \mathring{\mathbb B}_{\epsilon'})
\longrightarrow \mathbb S^1_\delta(0)
\end{equation}
is diffeomorphic to the fibration
\begin{equation}
\label{eq:milnor0} \Phi_W : \big( \mathbb S_\epsilon \setminus (L_f \cup L_g)\big) \cap F^{-1}(W)
\longrightarrow \mathbb S^1.
\end{equation}
where $W= \mathbb D_r(0) \setminus \mathring {\mathbb
D}_\delta(0)$.
\end{theorem}

Remember that $\check{\phi_0}$ is a restriction of $F$ and $\Phi_W$ is a restriction of $\frac{F}{|F|}$.

\begin{corollary}
\label{cor:1}
Let $f,g : (\mathbb C^n,0) \rightarrow (\mathbb
C,0)$ be two germs of holomorphic functions without common
branches such that $F=f/g$ is semitame and (i)-tame. For $\delta \ll 1$ and $R\gg1$ one has:

\begin{enumerate}

\item[\bf a)] The truncated global
Milnor fiber $\check{\mathcal M}_F = {\check{\Phi}_F}^{-1}(1)$
is diffeomorphic to  the union of the two restricted local Milnor
fibers  $\check{\mathcal M}_F^0 = \check{\phi}_0^{-1}(\delta)$
and $ \check{ \mathcal M}_F^{\infty} =
\check{\phi}_{\infty}^{-1}(R)$ glued along their
boundary components $\partial_0 =
\check{\phi}_0^{-1}(\delta) \cap \mathbb S_{\epsilon'} $
and $\partial_{\infty} = \check{\phi}_{\infty}^{-1}(R)
\cap \mathbb S_{\epsilon'} $
$$\check{\mathcal M}_F  \simeq \check{\mathcal M}_F^0 \cup_{\partial} \check{\mathcal M}_F^{\infty} \,.$$

\item[\bf b)]
The Euler characteristics verify:

$$
\chi\left(\check{\mathcal M}_F  \right) = \chi\left( \check{\mathcal M}_F^0\right)
+  \chi\left(  \check{\mathcal M}_F^{\infty} \right).
$$

and

$$
\chi\left( {\mathcal M}_F  \right) = \chi\left(  {\mathcal M}_F^0\right)
+  \chi\left(   { \mathcal M}_F^{\infty} \right).
$$

\item[\bf c)]
The monodromies $\check{h}_{0} : \check{\mathcal M}_F^0 \rightarrow \check{\mathcal M}_F^0$ and
$\check{h}_{\infty} : \check{\mathcal M}_F^{\infty} \rightarrow \check{\mathcal M}_F^{\infty}$ of the fibrations $\check{\phi}_{0}$ and $\check{\phi}_{\infty}$ are the restrictions of the monodromy $\check{h} : \check{\mathcal M}_F  \rightarrow \check{\mathcal M}_F $ of the fibration $\check{\Phi}_F$.
\end{enumerate}

\end{corollary}

\begin{proof}[Proof of the Corollary]
We apply Theorem \ref{th:fib0} twice with $r=1$. The first time as
stated, the second time around $\infty$, or  in other words,
around $0$ for $ g/f$. The proof of  Theorem \ref{th:fib0}
furnishes:
\begin{itemize}
\item a diffeomorphism $\Theta_0$
from
$$\check{\phi}_0^{-1}(\delta) = F^ {-1}(\delta) \cap \big(  {\mathbb B}_\epsilon
\setminus {\mathbb B}_{\epsilon'})$$
 to
$${\frac{F}{|F|}}^{-1}(1) \cap { \mathbb S}_\epsilon \cap F^{-1}({\mathbb D}_1(0) \setminus \mathring
{\mathbb D}_\delta(0)),$$
such that $\Theta_0 ^{-1}({\mathbb D}_1(0)) = \partial_0$,
\item and a diffeomorphism $\Theta_{\infty}$ from
$$\big(\check{\phi}_0\big)^{-1}(\delta) = F^ {-1}(R) \cap \big(  {\mathbb B}_\epsilon
\setminus {\mathbb B}_{\epsilon'})$$
 to
$${\frac{F}{|F|}}^{-1}(1) \cap {\mathbb S}_\epsilon \cap F^{-1}({\mathbb D}_R(0) \setminus
\mathring {\mathbb D}_1(0)),$$
such that $\Theta_{\infty} ^{-1}({\mathbb D}_1(0)) = \partial_{\infty}$.
\end{itemize}

The intersection  of the images of $\Theta_0$ and $\Theta_1$ is exactly
$$ \Theta_0(\partial_0) =   \Theta_{\infty}(\partial_{\infty}) = F^{-1}(1) \cap \mathbb S_{\epsilon}$$

Then  $\check{\mathcal M}_F = \check{\Phi}_F^{-1}(1) $ is diffeomorphic to the union of
$\check{\phi}_0^{-1}(\delta) $ and of $\check{\phi}_{\infty}^{-1}(R)$ glued
along their boundary components
$\partial_0 $ and $\partial_{\infty}$. This proves statement a).

\bigskip

The Euler characteristic verifies $\chi (A\cup B) =
\chi(A)+\chi(B)-\chi(A\cap B).$ As the intersection of the images
of $\Theta_0$ and $\Theta_1$ is a closed oriented manifold of odd
dimension, then its Euler characteristic is $0$. This proves the first equation in statement b). For the second equation, notice  ${\mathcal M}_F^0$ (respectively $ {\mathcal M}_F^{\infty}$)
retracts by deformation to $\check{\mathcal M}_F^0$ (respectively
$\check{\mathcal M}_F^{\infty}$), and ${\mathcal M}_F  $ retracts by
deformation to $\check{\mathcal M}_F  $, proving b).

The statement c) follows from a) and Theorem \ref{Result}.
\end{proof}

\begin{corollary}
\label{cor:2}   Let $f,g : (\mathbb C^n,0) \rightarrow (\mathbb
C,0)$ be two germs of holomorphic functions without common
branches such that $F=f/g$ is semitame and (i)-tame. If $f,g$ have an isolated singularity at $0$, then
$$ \chi\left({\mathcal M}_F\right) = (-1)^{n-1} \big( \mu(f,0)+\mu(g,0)-2\mu(f+tg,0) \big).$$
Where $t$ is a generic value ({\it i.e.},\ $t\neq 0,\infty$) and
$\mu$ is the Milnor number.
\end{corollary}

\begin{proof}
According to  \cite[Theorem 2]{GLM1},
$$\chi\left( {\mathcal M}_F^0 \right) =  (-1)^{n-1} \big( \mu(f,0)-\mu(f+tg,0) \big)$$
and
$$\chi\left( {\mathcal M}_F^{\infty} \right) =  (-1)^{n-1} \big( \mu(g,0)-\mu(f+tg,0) \big).$$
\end{proof}

\begin{corollary}
If $n=2$,  then the manifold ${\mathcal M}_F^0$ (respectively  $ {\mathcal
M}_F^{\infty} $) has the homotopy type of a  bouquet of  circles.
If we denote by $\lambda_0$ (respectively $\lambda_\infty$) the
number of circles in this bouquet. Then ${\mathcal M}_F$ is a bouquet
of $\lambda_0+\lambda_\infty+1$ circles.
\end{corollary}

%%%%%%%%%%%%%%%%%%%%%%%%%%%%%%%%%%%%%%%%%%%%%%%%%%%%%%%%%%%
%%%%%%%%%%%%%%%%%%%%%%%%%%%%%%%%%%%%%%%%%%%%%%%%%%%%%%%%%%%
\section{Preliminary lemmas}

The following lemmas are easily obtained by adapting the proofs of Lemmas 4.3 and 4.4  in \cite{Mi} as already performed in \cite{NZ} and \cite{BP} in close situations.

\begin{lemma}
\label{lem:arg} Assume that the meromorphic germ $F=f/g$ is semitame
at the origin. Let $p:[0,1] \longrightarrow \Cc^n$
be a real analytic path with $p(0)=0$ such that for all $t>0$, $F(p(t))
\notin \{0,\infty\}$ and such that the vector $\grad \log F(p(t))$ is
a complex multiple $\lambda(t)p(t)$ of $p(t)$. Then the argument of the complex
number $\lambda(t)$ tends to $0$ or $\pi$ as $t\rightarrow 0$.
\end{lemma}

\begin{proof}
Adapting  \cite[Lemma 4.4]{Mi}. See also  \cite[Lemma 3]{NZ} and  \cite[Lemma 2.7]{BP}.
\end{proof}

\begin{lemma}
\label{lem:indep}
Let $F$ be semitame. Then there exists $0<\epsilon \ll 1$ such that
for all $z \in B_\epsilon \setminus (F^{-1}(0) \cup F^{-1}(\infty))$
the two vectors $z$ and $\grad \log F(z)$
are either linearly independent over $\Cc$
or $\grad \log F(z) = \lambda z$ with $|\arg(\lambda)| \in ]-\frac \pi 4,+\frac \pi 4[$.
\end{lemma}
\begin{proof}
Using Lemma \ref{lem:arg}. See \cite[Lemma 4.3]{Mi} and \cite[Lemma 4]{NZ}.
\end{proof}

\begin{lemma}
\label{lem:twodisks} Let $D',D''$ be two 2-discs centered at $0$
with $D' \subset D''$ and $D' \neq D''$. For $0<\epsilon \ll 1$,
if $z \in \mathbb S_\epsilon \setminus (F^{-1}(0) \cup F^{-1}(\infty))$ is
such that $\grad \log F(z)= \lambda z$, ($\lambda \in \Cc$) then
$$F(z) \in D' \qquad \text{ or } \qquad F(z) \notin D''.$$
Moreover in the first case
$\arg(+\lambda) \in ]-\frac \pi 4,+\frac \pi 4[$
and in the second case
$\arg(-\lambda) \in ]-\frac \pi 4,+\frac \pi 4[$.
\end{lemma}

\begin{proof}
Using Lemma \ref{lem:arg} and Lemma \ref{lem:indep}. See \cite[Lemma 8]{NZ}.
\end{proof}

%%%%%%%%%%%%%%%%%%%%%%%%%%%%%%%%%%%%%%%%%%%%%%%%%%%%%%%%%%%
%%%%%%%%%%%%%%%%%%%%%%%%%%%%%%%%%%%%%%%%%%%%%%%%%%%%%%%%%%%
\section{Proof of the theorem}

\noindent
\textbf{First step: definition of the constants.}

\begin{enumerate}
  \item Let $0<r<\infty$ be such that $B \cap {\mathbb D}_r(0) = \{0\}$, where $B$ is the
  bifurcation set of the meromorphic function $f/g$.
  \item As $\mathbb S^1_r(0)$ is compact and $\mathbb S^1_r(0) \cap B =\emptyset$,  one can choose
  $0<\epsilon \ll 1$ such that:

  \begin{enumerate}

\item[a)] $\mathbb B_\epsilon$ is a Milnor ball for $F^{-1}(0),
F^{-1}(\infty)$, the indetermination set $I$ and for all $F^{-1}(z), z \in \mathbb
S^1_r(0)$;

\item[b)] $\epsilon$ satisfies the conclusion of Lemma
\ref{lem:twodisks} for $  D' = \mathbb D_{r/4}(0)$ and
$  D'' = \mathbb D_{r+1}(0)$.
\end{enumerate}

  \item  Let us  choose $\delta$, $0< \delta \ll \epsilon$, such
  that:
$$ \phi_0 = F_{\mid } : F^{-1} (\mathbb S^1_\delta(0)) \cap \mathbb B_\epsilon \longrightarrow
\mathbb S^1_\delta(0)$$ is the $0$-Milnor fibration of the
meromorphic map $F$.

\item Last, let us choose $\epsilon'_0, 0 < \epsilon'_0 \ll \delta$
such that $$M(F)\cap F^{-1}(\mathbb D_r(0) \setminus
\mathring{\mathbb D}_{\delta}(0)) \cap  {\mathbb B}_{\epsilon'_0}
=\emptyset,$$ and let us set $\epsilon'=\epsilon'_0 /2$.  In
particular, one obtains the restricted $0$-local Milnor fibration

$$\check{\phi_0} : F^{-1} (\mathbb S^1_\delta(0)) \cap (\mathbb B_\epsilon \setminus
\mathring{\mathbb B}_{\epsilon'}) \longrightarrow \mathbb
S^1_\delta(0) \,.
$$

 \end{enumerate}

Let $\psi : U \rightarrow \mathbb R^2$ be defined by $\psi(z)= (\log
|f(z)|, ||z||^2)$. Notice that Conditions 2.b)  and  4)  imply that
$$\big([\log{\delta}, \log r] \times [0,{\epsilon'_0}^2] \cup [\log (r/2), \log r]
\times [0,\epsilon^2] \big)\cap \psi (M(F)) = \emptyset \,.$$

\begin{figure}
\includegraphics{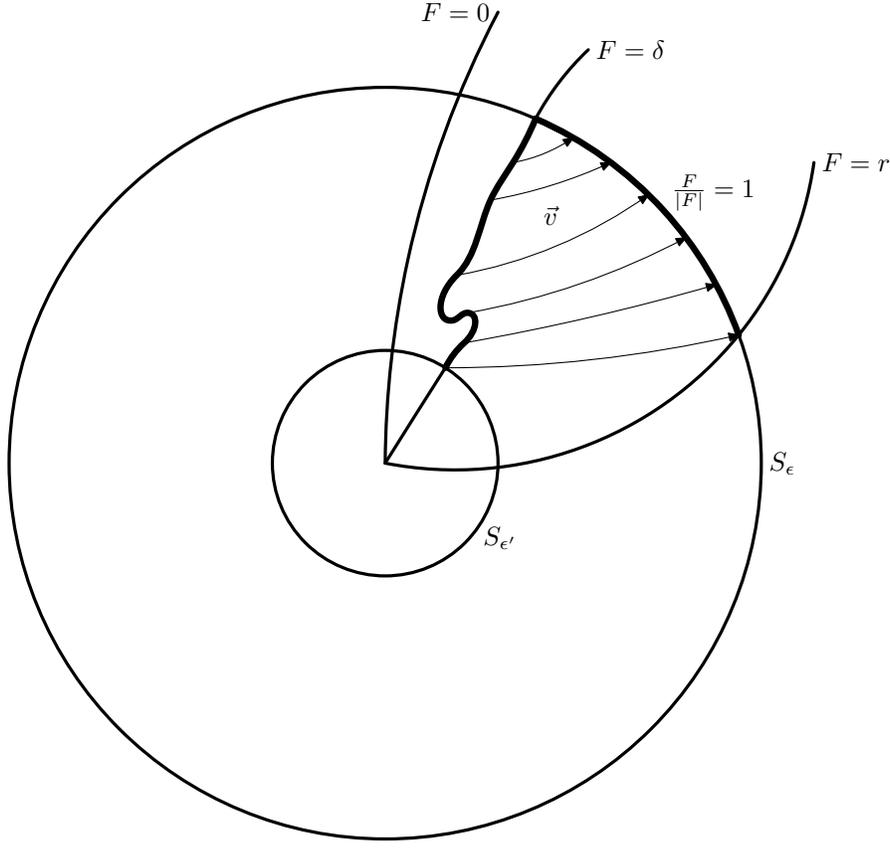}
\caption{\label{fig:mil1} Vector field}
\end{figure}

\noindent
\textbf{Second step : construction of a vector field.}

 Let us consider the set
 $$X = F^{-1} \Big( {\mathbb D}_r(0)\setminus \mathring {\mathbb D}_\delta(0)  \Big)
 \cap \big( {\mathbb B}_\epsilon \setminus  \mathring {\mathbb B}_{\epsilon'} \big).
 $$
Let us fix $\rho \in ]r/2,r[$ and  $ \epsilon'' \in
]\epsilon',\epsilon_0[$ and let us consider the two real numbers
$0<b_1 <b_2$ defined by:
$$ b_1 = \frac{\epsilon''^2-\epsilon'^2}{2(\log \rho - \log \delta)} \hskip0.5cm  \hbox{ and }
\hskip0.5cm b_2 =  \frac{\epsilon_0^2 -\epsilon''^2}{2(\log r-
\log \rho) } \;.$$
%$$K = \frac{\epsilon^2 - {\epsilon'}^2}{\big(\log r-\log \delta +1\big)^n-1}$$
Let us fix $\xi, 0<\xi  \ll \epsilon''$ and an increasing $C^{\infty}$ map $b : [0 ,\infty[ \rightarrow [0 ,\infty[ $ such that $\forall x \leq \epsilon"-\xi, b(x) = b_1$ and $\forall x \geq \epsilon"+\xi, b(x) = b_2$.

 For $\eta >0$,  we consider  the neighbourhood of $I$ defined by:
   $$N_{\eta} =\big\{ z \in \mathbb B_{\epsilon} \ | \ |f(z)|^2 + |g(z)|^2 \leq \eta^2\big\},$$

   and its boundary, $$\partial N_{\eta}     =
   \big\{ z \in \mathbb B_{\epsilon} \ | \ |f(z)|^2 + |g(z)|^2 = \eta^2\big\}.$$

Let us fix $\eta$, $0 < \eta \ll \epsilon'$, such that $\eta$
satisfies the (i)-tameness condition :
$$ \big( N(F) \cap N_{\eta} \cap X \big) \subset  \big( M(F) \cap N_{\eta} \cap X \big)$$

\begin{lemma}  There exists
an  open  neighbourhood $\Omega$ of the set $M(F)$  in $X$, two real
numbers $\alpha$ and $ \beta$, $0<\alpha< \beta$ and a
differentiable  vector field $v$ on $X$ such that:

\begin{enumerate}
\item[(i)] For all $z\in X, \langle v(z) , \grad \log F(z)
\rangle = +1 $,

\item[(ii)] For all $z\in X \setminus M(F), \langle v(z) , z \rangle = b(|z|)$

\item[(iii)] For all $z\in   \Omega, Re \langle v(z) ,  z \rangle \in
[\alpha,\beta]$;

\item[(iv)] For all $z \in X \cap  N_{\eta}$, $v(z) \in
T_z \partial N_{\eta'}$ where ${\eta'} ^2= |f(z)|^2 + |g(z)|^2 $.
\end{enumerate}
\end{lemma}

\begin{proof} Let $\mu, 0<\mu  \ll \eta$, be such that $\eta + \mu$ again satisfies
 the (i)-tameness condition :
$$ \big( N(F) \cap N_{\eta + \mu} \cap X \big) \subset  \big( M(F) \cap N_{\eta + \mu}
\cap X \big).$$

Let us denote by $V$ the interior of $N_{\eta + \mu}$ in $X$, {\it
i.e.},
$$V =  \{ z \in X \  |  \ 0  \leq |f(z)|^2 + |g(z)|^2 < (\eta + \mu)^2 \},$$
and let us consider the four following open sets of $X$ (the neighbourhood $\Omega$ of $M(F)$
will be defined later) :
$$U_1 =  X\setminus (N_{\eta} \cup M(F)), \qquad U_2 = \Omega \setminus N_{\eta}$$
$$ U_3 =   V \cap \Omega,  \qquad U_4 = V \setminus M(F) \,.$$
One has : $X = U_1\cup U_2\cup U_3\cup U_4$. The vector field $v$ will be obtained by
constructing a vector field $v_i$ on each  $U_i$  and by defining globally $v$ by
 a partition of unity.

At first, let us define $v$ on $X \setminus
N_{\eta} = U_1 \cup U_2$.
For a point $z
\in U_1$, we define $v_1$  by using the classical construction of
Milnor: for such a point the vectors $z$ and $\grad \log F(z)$ are
linearly independent over $\mathbb C$. Thus there exists $v_1(z)$
verifying (i) and (ii).

For each $z \in X$, let us consider the vector
$$u(z) =  \frac{ \grad \log F(z) }{|| \grad \log F(z)||^2} \;.$$

Let $z\in M(F) \cap X$. There exists $\lambda \in \Cc$ such that
$\grad \log F(z) = \lambda z$. Then  $ \langle u(z) , \grad
\log F(z) \rangle = +1$ and

$$\Re \langle u(z) , z \rangle = \Re\left( \frac{\lambda}{|\lambda|^2}\right) \;.$$

Notice that $M(F) \cap \mathbb B_{\epsilon} = \left\lbrace z \in
\mathbb B_{\epsilon} \mid \exists \lambda \in \Cc,
  \grad F (z)=\lambda z \right\rbrace$  is compact. Then $M(F) \cap X$ is a compact set,
  and there exist  $c_1,c_2$, $0<c_1<c_2$
such that for all $z\in M(F)\cap X$  one has $c_1 < |\lambda| <
c_2$ where $\lambda$ is the complex number such that  $\grad F
(z)=\lambda z $. Moreover, Condition 2.b) below
  implies that
 $\arg(+\lambda) \in ]-\frac \pi 4,+\frac \pi 4[$. Then there exists  $c'_1,c'_2 >0$ such that
for all $z\in M(F)\cap X$,
 $c'_1 < \Re \langle u(z) , z \rangle< c'_2$.

 Let us choose $\nu$ such that $0<\nu  \ll  c'_1$ and let us set  $\alpha = c'_1-\nu$ and
 $\beta = c'_2 + \nu$. There exists an open neighbourhood $\Omega$ of $M(F)$ in $X$ such that
 for all $z\in \Omega$,
 $\alpha < \Re \langle u(z) , z \rangle< \beta$.
Then for each $z \in U_2 = \Omega \setminus N_{\eta}$, we set $v_2(z) =
u(z)$.

We now define $v$ on $V = U_3 \cup U_4$, {\it i.e.} near the indetermination set $I$. A  picture of the local situation near $I$ is represented on Figure \ref{fig:mil2}.
\begin{figure}
  \includegraphics{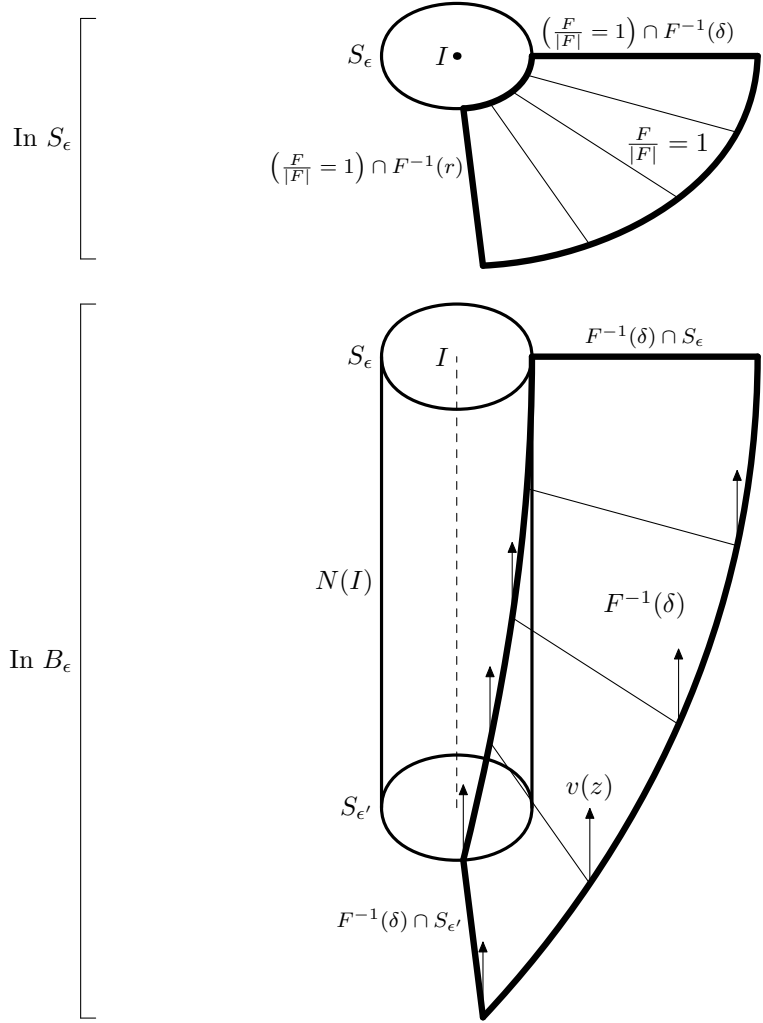}
  \caption{\label{fig:mil2} Vector field and indetermination points}
\end{figure}
For a point $z$ in $V$,  we set $$\eta' = \sqrt{|f(z)|^2 +
|g(z)|^2}\,.$$ Let $T = T(z)$ be the  space tangent to $\partial
N_{\eta'}$ at $z$. We will construct a vector field $v$ on $V$
satisfying the three conditions  (i),  (ii) and (iii) and such
that $v(z) \in T$.

If  $z \in U_3 = V \cap \Omega$, let us again consider the vector
$u(z)$. If $u(z) \in T$, then we set $v_3(z)=u(z)$. If $u(z)
\not\in T$, let
$$Q =  \big( u(z) \big)^{{\bot}_{\mathbb R}}\,,$$
be the real orthogonal complement of
the line spanned by the vector $u(z)$.

Since $\dim_{\mathbb R} Q = 2n-1$ and $\dim_{\mathbb R} T = 2n-1$,
the real vector space $Q  \cap T$ has  dimension at least $n-2$.
Let $\pi : {\mathbb R}^{2n} \rightarrow Q \cap T$ be the
orthogonal projection on $Q \cap T$ in the direction of the vector
$i u(z)$. We set
$$v_3(z)=\pi(u(z))\,.$$
Obviously $v_3(z) \in T$ and an easy computation shows that $v_3$
verifies conditions (i) and (iii).

Last, let us consider  $U_4 = V \setminus M(F)$. Let  $z \in U_4$.
We set $\gamma(z)= |f(z)|^2 + |g(z)|^2$. There exists a vector
$v_4(z)$ verifying (i), (ii) and $v_4(z) \in T$ if and only if the
vector
$$w(z)=\grad_{\mathbb R} \gamma(z)$$
does not belong to the complex vector space $H$ generated by the
two vectors $w_1(z) = z$ and $w_2 (z)= \grad F(z)$.  This is
equivalent to saying $z \not\in N(F)$, which is true because $F$
is  semitame and (i)-tame.

\medskip

 Now, we define globally the vector field on $X$ by a
partition of unity.

\end{proof}

\bigskip

\noindent
\textbf{Third step : integration of the vector field $v$.}

We integrate the vector field $v$ and we denote by $C=\{ z=p(t)\}$
an integral curve.

Condition (i) implies that the argument of $F(p(t))$ is constant
and that and that $|F(p(t))|$ is strictly increasing along $C$.
Conditions (ii) and (iii) implies that $\|p(t)\|$ is strictly
increasing along $C$.

\begin{lemma} If $C$  pass through a point  $z_0 \in F^{-1} (\mathbb S^1_\delta(0)) \cap
( \mathbb S_{\epsilon'})$, then $C$ reaches $\mathbb S_{\epsilon}$ at a point $z_1$ such that $|F(z_1)| = r$.
\end{lemma}

\begin{proof} Let $C'$ be the arc of curve in $\mathbb R^2$ parameterized by $t \in [0, \log(r/\delta)]$
as follows :
\begin{itemize}
\item $x(t)=t + \log \delta$
\item $\forall t \in [0,
\log(\rho/\delta)],  y(t) = 2b_1t + \epsilon'^2 $
\item $\forall t \in [ \log(\rho/\delta), \log(r/\delta)],  y(t) = 2b_2t
+ \epsilon''^2 $
\end{itemize}

The arc $C'$ is the union of the two segments joining the three
points $(\log \delta, {\epsilon'}^2)$,  $(\log \rho,
{\epsilon''}^2)$ and $(\log r, {\epsilon}^2)$. Then  $C'$ is
included in the zone $ P = [\log{\delta}, \log r] \times
[0,{\epsilon'_0}^2] \cup [\log (r/2), \log r]  \times
[0,\epsilon^2]$ and  $C' \cap \psi(U) = \emptyset$. (see Figure
\ref{fig:mil3}).
\begin{figure}
  \includegraphics{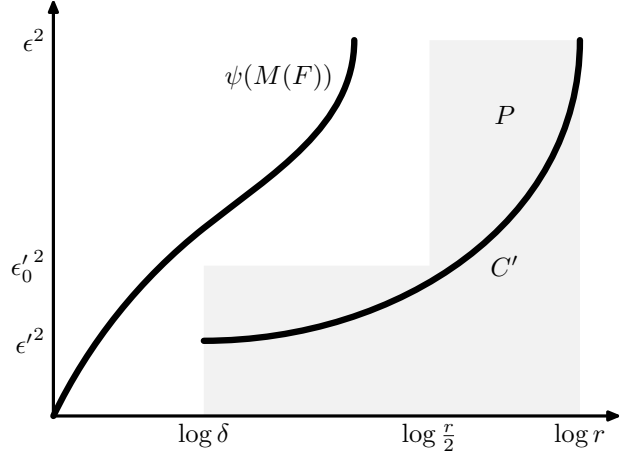}
  \caption{\label{fig:mil3} Avoidance of $M(F)$}
\end{figure}
Now, let $C$ be an integral curve of $v$ passing through $z_0 \in
F^{-1} (\mathbb S^1_\delta(0)) \cap \mathbb  S_{\epsilon'}$. Then a computation
analogous to that of \cite{Mi} page 53, shows that  $C'$ is nothing
but the image of $C$ by  $\psi$.  Therefore, the integral curve
$C$ passing through $z_0$ goes transversally to the spheres
centered at $0$ until it reaches $\mathbb S_\epsilon$ at a point belonging
to $ F^{-1}(\mathbb S^1_r(0))$.
\end{proof}

Then, the diffeomorphism $$\Theta_0 : F^{-1} (\mathbb S^1_\delta(0)) \cap
(\mathbb B_\epsilon \setminus \mathring {\mathbb B}_{\epsilon'})
\longrightarrow \mathbb S_\epsilon \cap F^{-1}(\mathbb D_r(0)
\setminus \mathring {\mathbb D}_\delta(0)),
$$

\noindent which sends $z \in F^{-1} (\mathbb S^1_\delta(0)) \cap
(\mathbb B_\epsilon \setminus \mathring {\mathbb B}_{\epsilon'}) $
on the intersection $\Theta_0(z)$ of the integral curve of $v$
passing through $z$ with the sphere  $S_\epsilon \cap F^{-1}(
\mathbb D_r(0) \setminus \mathring {\mathbb D}_\delta(0))  $,  is
a diffeomorphism from the fibration:
$$
F : F^{-1} (\mathbb S^1_\delta(0)) \cap (\mathbb B_\epsilon
\setminus \mathring {\mathbb B}_{\epsilon'}) \longrightarrow
\mathbb S^1_\delta(0)
$$
to the fibration:
\begin{equation}
\label{eq:milnorout0} \Phi = \frac{F}{|F|} : \mathbb S_\epsilon \cap
F^{-1}( {\mathbb D}_r(0) \setminus \mathring {\mathbb D}_\delta(0))  \longrightarrow
\mathbb S^1.
\end{equation}

This completes the proof of the theorem.

%%%%%%%%%%%%%%%%%%%%%%%%%%%%%%%%%%%%%%%%%%%%%%%%%%%%%%%%%%%
%%%%%%%%%%%%%%%%%%%%%%%%%%%%%%%%%%%%%%%%%%%%%%%%%%%%%%%%%%%
\section{An example}

Let $f,g : ({\mathbb C^2},0) \longrightarrow ({\mathbb C},0)$ be the two holomorphic germs  defined by:
$$f(x,y)=x^2 + y^3, \   \  g(x,y)=x^2 + y^3.$$
Let $\pi : X \longrightarrow U$ be the resolution of the meromorphic function $F=f/g$ whose divisor is represented on Figure \ref{fig:shortexample}.
\begin{figure}
\includegraphics{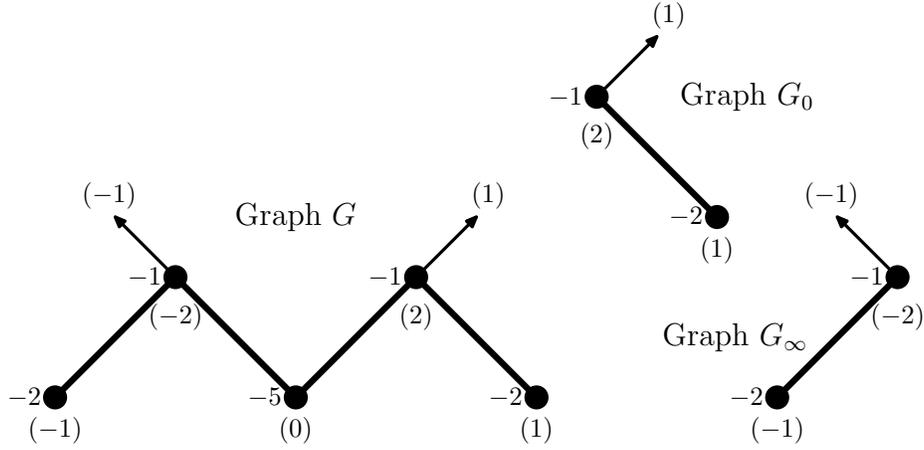}
\caption{\label{fig:example} Resolution  of $x^2+y^3/x^3+y^2/g$}
\end{figure}
On Figure \ref{fig:example} we draw its dual graph $G$. The numbers between parentheses are the multiplicities of $F$  along the corresponding component of the total transform of $fg$ by $\pi$, {\it i.e.}, the $(m_i^f - m_i^g)$ where $m_i^f$ and $m_i^g$ are the multiplicities  of $f\circ \pi$ and $g \circ \pi$. The numbers without parentheses are the Euler classes.
The arrows are for the strict transforms of $f$ and $g$ and for the strict transform of a generic fibre of $F$.

The meromorphic function $f/g$ is semitame, and $(i)$-tame ($n=2$). One therefore has three different fibrations:
the global Milnor fibration of $f/g$,
$$\Phi_F =  \frac{f/g}{|f/g|} : \mathbb S_\epsilon \setminus (L_f \cup L_g)
 \longrightarrow \mathbb S^1,$$
 and the two local Milnor fibrations $ \phi_0 = F_{\mid } : F^{-1} ({\mathbb S}^1_\delta(0)) \cap {\mathbb B}_\epsilon \longrightarrow {\mathbb S}^1_\delta(0) $ and $ \phi_{\infty} = F_{\mid } : F^{-1} ({\mathbb S}^1_\delta(\infty)) \cap {\mathbb B}_\epsilon \longrightarrow {\mathbb S}^1_\delta(\infty) $

 Using the fibration theorem   for plumbed multilinks (see e.g. \cite[2.11]{PS}), one observes three different fibred multilinks in plumbing manifolds on this configuration :

\begin{enumerate}
  \item The link $L_f - L_g$ in the sphere $\mathbb S^3$.
  \item The link $L_f$ in the plumbed manifold $V_0$ whose graph $G_0$ is the subgraph of $G$ determined by the divisor $E_2 \cup E_3$.
  \item The link $L_g$ in the plumbed manifold $V_{\infty}$ whose graph $G_{\infty}$ is the subgraph of $G$ determined by  the divisor $E_3 \cup E_4$.
 \end{enumerate}

 As already mentioned, the map $\Phi_F$ is  a fibration of the link $L_f - L_g$ in the sphere $\mathbb S^3$.  The two local fibrations  $\phi_0$ and $\phi_{\infty}$  are the restrictions to the complementary of the indetermination set $I$ of $f/g$ of fibrations $\hat{\phi}_0$ and  $\hat{\phi}_{\infty}$ of the links $L_f \subset V_0$ and  $L_g \subset V_{\infty}$.

The fibres $\hat{\mathcal M}_F^0$ and $\hat{\mathcal M}_F^{\infty}$ of $\hat{\phi}_0$ and  $\hat{\phi}_{\infty}$ can be computed by the Hurwitz formula from the graphs $G_0$ and $G_{\infty}$. One obtains for both a sphere with one hole. The fibre ${\mathcal M}_F^0$ (resp. ${\mathcal M}_F^{\infty}$) is then obtained by  removing  a neighbourhood of the intersection of $\hat{\mathcal M}_F^0$ with $\pi^{-1}(0)$. One then obtains a sphere with 3 holes in both cases. Now the fiber of $\Phi_F$ is homeomorphic to the surface obtained by glueing together ${\mathcal M}_F^0$ and ${\mathcal M}_F^{\infty}$ along the two boundary components just created.

At last, let us recall that the  isomorphism classes of the fibrations $\Phi_F, \hat{\phi}_0$ and $\hat{\phi}_{\infty}$ are completely described by the Nielsen invariants of their monodromies $h: {\mathcal M}_F \rightarrow {\mathcal M}_F $, $\hat{h}_0 : \hat{\mathcal M}_F^{0} \rightarrow \hat{\mathcal M}_F^{0}$ and
$\hat{h}_{\infty} : \hat{\mathcal M}_F^{\infty} \rightarrow \hat{\mathcal M}_F^{\infty}$ (see e.g. \cite{P1}). The set of Nielsen invariants is equivalent to the data of the graph $G, G_0$ and $G_{\infty}$ respectively, weighted by the genus and the multiplicities.

In particular, it should  be mentioned that the Dehn twist performed by $h$ on the two
glueing curves $\partial_0 \subset {\mathcal M}_F$ is positive (it equals $+\frac{5}{2}$
according, for instance,   to the formula (3) of \cite[Lemma 4.4]{P1}), whereas all the Dehn twists performed by the monodromy of the Milnor fibration associated with a holomorphic germ are negative. This is a general phenomena in the case $n=2$ : each dicritical component of the exceptional divisor in the resolution of $f/g$ gives rise to a positive Dehn twist along each of the corresponding separating curves on the fiber, see \cite[Chapter 5]{Bo}.

The previous arguments show that in this example,  the monodromy
of the global Milnor fibration of the meromorphic germ $f/g$ can
not possibly be the monodromy of a holomorphic germ. Then,   a
natural question is to ask whether something similar happens in
higher dimensions. That is, is there some property that
distinguishes the   monodromy of the global Milnor fibration of a
meromorphic function from those of holomorphic germs?

%%%%%%%%%%%%%%%%%%%%%%%%%%%%%%%%%%%%%%%%%%%%%%%%%%%%%%%%%%%%%%%%%
%%%%%%%%%%%%%%%%%%%%%%%%%%%%%%%%%%%%%%%%%%%%%%%%%%%%%%%%%%%%%%%%%

Arnaud Bodin (\texttt{Arnaud.Bodin@math.univ-lille1.fr})

Laboratoire Paul Painlev\'e,

Universit\'e de Lille 1,

59655 Villeneuve d'Ascq, France.

\bigskip
Anne Pichon (\texttt{pichon@iml.univ-mrs.fr})

 Institut de Math\'ematiques de Luminy,

 UMR 6206 CNRS, Campus de Luminy Case 907,

 13288 Marseille Cedex 9, France.

\bigskip

Jos\'e Seade (\texttt{jseade@matcuer.unam.mx}),

Instituto de Matem\'aticas, Unidad Cuernavaca,

Universidad Nacional Aut\'onoma de M\'exico,

Av. Universidad s/n, Lomas de Chamilpa,

62210 Cuernavaca, Morelos, A. P. 273-3,
 M\'exico.

\end{document}